\documentclass[reqno]{amsart}

\usepackage[sorting=none, sorting=nyt, maxbibnames=99, backend=biber]{biblatex}

\renewbibmacro{in:}{}
\bibliography{references.bib}
\usepackage{amssymb}
\usepackage{calrsfs}
\usepackage{graphics,graphicx,mathrsfs}
\usepackage{enumerate}
\usepackage{enumitem}
\usepackage{textcomp}
\usepackage{bbold}
\usepackage{url}
\usepackage{xcolor}
\definecolor{vio}{rgb}{0.54, 0.17, 0.89}
\usepackage[ruled, lined, linesnumbered, longend]{algorithm2e}
\usepackage{hyperref}
\usepackage{titlesec}
\usepackage{todonotes}
\usepackage[mathcal]{eucal}

\newtheorem{theorem}{Theorem}[section]
\newtheorem{lemma}[theorem]{Lemma}

\newtheorem{proposition}[theorem]{Proposition}

\newtheorem{corollary}[theorem]{Corollary}
\numberwithin{equation}{section}

\theoremstyle{remark}
\newtheorem*{remark}{Remark}

\theoremstyle{definition}
\newtheorem*{definition}{Definition}

\titleformat{\section}
  {\normalfont\large\bfseries\centering}{\thesection}{1em}{}

\titleformat{\subsection}
  {\normalfont\bfseries}{\thesubsection}{1em}{}  

\DeclareMathOperator{\Z}{\mathbb{Z}}

\DeclareMathOperator{\N}{\mathbb{N}}
\DeclareMathOperator{\G}{\mathscr{G}}
\DeclareMathOperator{\li}{\mathrm{li}}
\DeclareMathOperator{\ord}{\mathrm{ord}}

\newcommand{\mcS}{\mathcal{S}}

\newcommand{\mcA}{\mathcal{A}}
\newcommand{\msA}{\mathscr{A}}
\newcommand{\mcE}{\mathcal{E}}
\newcommand{\mcT}{\mathcal{T}}

\def\reals{\hbox{\rm I\kern-.18em R}}
\def\complexes{\hbox{\rm C\kern-.43em
\vrule depth 0ex height 1.4ex width .05em\kern.41em}}
\def\field{\hbox{\rm I\kern-.18em F}} 

\let\svthefootnote\thefootnote
\newcommand\freefootnote[1]{%
  \let\thefootnote\relax%
  \footnotetext{#1}%
  \let\thefootnote\svthefootnote%
}

\newenvironment{section*}[2][A]{
  \section*{#2}
  \renewcommand\thesection{#1}
  \setcounter{theorem}{0}}{}

\allowdisplaybreaks

\begin{document}

\title[The least prime with a given primitive root]{The least prime with a given primitive root and a variant of the larger sieve}

\author{Daniel R. Johnston}
\author{Sunil Naik}
\address{Max Planck Institute for Mathematics, Vivatsgasse 7, 53111 Bonn, Germany}
\email{johnston@mpim-bonn.mpg.de}
\address{Department of Mathematics, Indian Institute of Technology Palakkad, Kerala 678623, India}
\email{sunilnaik@iitpkd.ac.in}
\date\today

\begin{abstract}
    For an integer $g$ not equal to $-1$ or a square, we study the least prime $p_g$ such that $g$ is a primitive root modulo $p_g$. We show that, assuming the Generalised Riemann Hypothesis, all such $g$ with $|g|\in[N,2N]$ have $p_g\leq(\log |g|)^{3.44}$ with at most $O((\log N)^{2.44})$ exceptions. This can be directly compared to the uniform bound $p_g\leq(\log |g|)^{19}$ proven recently by Fan and Pollack. To obtain our result, we prove a new variant of Gallagher's larger sieve, which may be of independent interest. In addition to our conditional result, we discuss other ``almost all" bounds for $p_g$ that hold unconditionally.
\end{abstract}

\maketitle 

\freefootnote{\textit{Corresponding author}: Daniel Johnston (johnston@mpim-bonn.mpg.de).}
\freefootnote{\textit{Affiliation}: Max Planck Institute for Mathematics, Bonn, Germany.}
\freefootnote{\textit{Key phrases}: primitive root, generalised Riemann hypothesis, exceptional set, larger sieve.}
\freefootnote{\textit{2020 Mathematics Subject Classification}: 11A07, 11N36.}

\section{Introduction} 
\subsection{Overview and a conditional result}
For a prime number $p$, we call an integer $g$ a \emph{primitive root} modulo $p$, if $g$ generates the cyclic group $(\Z/p\Z)^{\times}$. If $p>3$, any such $g$ must belong to the set
\begin{equation*}
    \mathscr{G}:=\{g\in\Z:g\neq -1\ \text{and}\ g\ \text{not a square}\}.
\end{equation*}

A 1927 conjecture of Artin asks whether each $g\in\G$ is a primitive root for infinitely many primes $p$, and also suggests an asymptotic for the number of such primes $p\leq x$ for a fixed $g$. For a comprehensive survey on Artin's primitive root conjecture, see~\cite{moree2012}. Most notably, a celebrated result of Hooley~\cite{hooley1967} gives that Artin's primitive root conjecture holds under the assumption of the Generalised Riemann Hypothesis (GRH) for Dedekind zeta functions.

In the context of Artin's primitive root conjecture, it is natural to study the quantity $p_g$, defined as follows.
\begin{definition}
    For a given $g\in\G$, let $p_g$ be the smallest prime $p$ such that $g$ is a primitive root modulo $p$. If such a $p$ does not exist, then set $p_g=\infty$.
\end{definition}
The quantity $p_g$ can be viewed as the dual concept of the widely studied quantity $g_p$, which corresponds to the least primitive root\footnote{In the literature, further restrictions are often placed on $g$. For instance, it is common to try to bound the least \emph{prime} primitive root.} $g$ modulo a fixed prime $p$ (see e.g.~\cite{burgess1962,martin1997,martin1998,Shoup1992}).
Now, unconditionally, it is not known whether $p_g<\infty$ for all $g\in\G$. However, assuming GRH, Fan and Pollack~\cite{Fan2025} modified Hooley's techniques to show
\begin{equation}\label{eq19}
    p_g\ll (\log |g|)^{19}\quad\text{(assuming GRH)},
\end{equation}
for all $g\in\G$. We remark that~\eqref{eq19} is far from what should be possible. In particular, Fan and Pollack also provide heuristics to conjecture that $p_g\ll \log|g|\log\log|g|$. 

We also note that upon imposing conditions on $g$, Fan and Pollack's result~\eqref{eq19} improves further. For instance, they obtain a better bound of $(\log|g|)^{16}$ if $g$ is positive, and $(\log|g|)^{13}$ if $g$ is square-free (see~\cite[p.~20]{Fan2025}). Trivially, one also has $p_g=2$ whenever $g$ is odd.

Our main result, given as follows, shows that in fact one can substantially improve~\eqref{eq19} upon removing a very small subset of $\G$.

\begin{theorem}\label{conthm}
    Let $N$ be sufficiently large and fix $B>2+2e^{-1/3}=3.433\ldots\,.$ Then, assuming GRH, all $g\in\G$ with $|g|\in[N,2N]$ satisfy
    \begin{equation}\label{almostpgeq}
        p_g\leq (\log |g|)^{B}
    \end{equation}
    with at most $O\big((\log N)^{B-1}\big)$ exceptions.
\end{theorem}
Of course, Theorem~\ref{conthm} is at its strongest when $B$ is taken as close to $2+2e^{-1/3}$ as possible. So to simplify matters, we recast our result into the following neater, yet less precise, corollary.
\begin{corollary}
    Let $N$ be sufficiently large. Then, assuming GRH, all $g\in\G$ with $|g|\in[N,2N]$ satisfy
    \begin{equation}\label{almostpgeq2}
        p_g\leq (\log |g|)^{3.44}
    \end{equation}
    with at most $O\big((\log N)^{2.44}\big)$ exceptions.
\end{corollary}
As with Fan and Pollack's uniform result~\eqref{eq19}, small improvements to~\eqref{almostpgeq2} are possible if we restrict $g$. Namely, if we restrict to positive $g$, we obtain a bound of $(\log|g|)^{3.35}$ with an $O((\log N)^{2.35})$ exceptional set. Then, if we restrict to square-free $g$, we obtain a bound of $(\log |g|)^{3.22}$ with an $O((\log N)^{2.22})$ exceptional set. For more information on these calculations, see the remark at the end of Section~\ref{consect}.

The key new input to obtain Theorem~\ref{conthm} is a novel variant of Gallagher's larger sieve (Proposition~\ref{newlargersieve}). The original larger sieve of Gallagher~\cite{gallagher1971} estimates the number of integers in an interval after a large number of residue classes are removed. By contrast, our variant is more flexible, by also allowing one to include integers that belong to a \emph{positive proportion} of the remaining residue classes. This variant of the larger sieve is well adapted to the structure of Fan and Pollack's argument, and we believe that it may have further applications. 

\subsection{Unconditional results}
One can also ask whether an unconditional variant of Theorem~\ref{conthm} is achievable, given that we are not looking for a completely uniform bound on $p_g$. This problem has in fact been studied before, but requires a different approach and a much larger exceptional set.

We begin by recalling some past results from the literature. Note that the following unconditional results are stated for arbitrary intervals $[M,M+N]$ of consecutive integers, rather than the dyadic interval we used in Theorem~\ref{conthm}. Firstly, Vaughan~\cite{vaughan1973} showed that
\begin{equation}\label{vaughaneq}
    p_g\leq \sqrt{N}
\end{equation}
for all but $O\big(\sqrt{N}(\log N)^{1-\lambda}\big)$ integers $g$ in an interval $[M,M+N]$, where
\begin{equation*}
    \lambda=\sum_{h=1}^\infty\prod_p\left(1-\frac{p^h-(p-1)^h}{p^{h+1}-p^h}\right)=0.629\ldots\,.
\end{equation*}
Vaughan's result sharpens an earlier estimate of Gallagher~\cite{gallagher1967}, who obtained the same bound~\eqref{vaughaneq} but with an exceptional set of size $O(\sqrt{N}\log N)$.

Fan and Pollack then showed that if one enlarges the exceptional set further, a much better bound for $p_g$ is possible. In particular, they showed~\cite[Proposition~6.1]{Fan2025}, for any fixed $K\geq 2$
\begin{equation*}
    p_g\leq (\log N)^K,
\end{equation*}
for all but $N^{\frac{1}{2}\left(1+\frac{1}{K}\right)+o(1)}$ integers $g$ in $[M,M+N]$. So for example, with at most $N^{3/4+o(1)}$ exceptions, $p_g\leq (\log N)^2$ for each integer $g\in[M,M+N]$.

We contribute further along this line of investigation, by considering how small $p_g$ can be, if all that we desire is an exceptional set of size $o(N)$. Ultimately, we are able to obtain the following unconditional result.

\begin{theorem}\label{unconthm}
     Let $\Phi(x)$ be any function satisfying $\Phi(x)\to\infty$ as $x\to\infty$. Then, for all $g\in[M,M+N]$,
     \begin{equation*}
         p_g\leq \Phi(N)
     \end{equation*}
     with at most $o(N)$ exceptions.
\end{theorem}
\begin{remark}
    Taking more care in the proof of Theorem~\ref{unconthm}, one obtains a precise estimate for the exceptional set. Namely, if $\Phi(x)\leq (1-\epsilon)\log x$ for a fixed $\epsilon>0$ and sufficiently large $x$, then the size of the exceptional set is bounded above by
     \begin{equation*}
         N\exp\left(-(c+o(1))\frac{\Phi(N)}{\log \Phi(N)}\right)=o(N)
     \end{equation*}
     for some absolute constant $c>0$.
\end{remark}

At first glance, Theorem~\ref{unconthm} appears to be a very strong result. However, its proof is elementary, using basic facts about primitive roots combined with the Chinese remainder theorem. Nevertheless, Theorems~\ref{conthm} and~\ref{unconthm} both highlight a general principle: \emph{for almost all $g\in\G$, one can quickly find a prime $p$ such that $g$ is a primitive root modulo $p$}.

\subsection{Outline of paper and notation}
The structure of the rest of the paper is as follows. In Section~\ref{largersect}, we introduce Gallagher's larger sieve and then prove our ``positive proportion" variant. In Section~\ref{consect}, we prove Theorem~\ref{conthm} by modifying Fan and Pollack's argument, and applying our larger sieve. Finally, we give a short proof of Theorem~\ref{unconthm} in Section~\ref{unconsect}.

Throughout we use Vinogradov's notation $f\ll g$ to mean $f=O(g)$, and the notation $f\asymp g$ to mean that both $f\ll g$ and $g\ll f$. Further, we write $\ord_p(g)$ for the multiplicative order of $g$ modulo $p$, $\Lambda(n)$ for the von-Mangoldt function, $\mu(n)$ for the M\"obius function, $\pi(x)$ for the prime-counting function,
\begin{equation*}
    \li(x):=\int_0^x\frac{1}{\log t}\mathrm{d}t\sim\frac{x}{\log x}
\end{equation*}
for the logarithmic integral, and $(a,b)$ as a shorthand for $\gcd(a,b)$.

\section{A positive proportion variant of the larger sieve}\label{largersect}
In this section we prove a new variant of Gallagher's larger sieve (Proposition~\ref{newlargersieve}), which will be an essential ingredient in the proof of Theorem~\ref{conthm}. We begin by recalling the standard version of the larger sieve, which estimates the size of a set of integers after a large number of residue classes have been removed.
\begin{proposition}[{Gallagher's larger sieve \cite[Theorem~1]{gallagher1971}}]\label{largerprop}
    Let $N\in\mathbb{N}$, and $Q$ be a finite set of prime powers. Suppose that all but $\nu(q)$ residue classes modulo $q$ are removed for each $q\in Q$. Then, among a set of $N$ consecutive integers, the number remaining unsieved does not exceed
    \begin{equation}\label{orginallargebound}
        \frac{\sum_{q\in Q}\Lambda(q)-\log N}{\sum_{q\in Q}\frac{\Lambda(q)}{\nu(q)}-\log N}
    \end{equation}
    provided the denominator is positive.
\end{proposition}

Although a rather general statement, Gallagher's larger sieve is particularly useful in the study of primitive roots, or more broadly the multiplicative order of elements modulo $p$. Notably, Gallagher~\cite[Theorem~2]{gallagher1971} used his larger sieve to show that the number of $g\leq N$ with $\ord_p(g)\leq N^{\theta}$ for all primes $p\leq N^{\theta+\epsilon}$ is $\ll N^{\theta}$, uniformly for $0\leq\theta\leq 1$. In a closer connection to our result, Fan and Pollack~\cite[Theorem~7.1]{Fan2025} used Gallagher's larger sieve to study the least prime $p$ such that a given $g\in \G$ is an\footnote{Here, Fan and Pollack define an \emph{almost primitive root} $g$ modulo $p$ to be such that $g$ generates a subgroup of $(\Z/p\Z)^{\times}$ of index at most 2.} ``almost-primitive root" modulo $p$. 

Now, in the language of Proposition~\ref{largerprop}, we let $\Omega_q$ denote the set of unsieved residue classes, so that $\nu(q)=\#\Omega_q$. Then, Gallagher's larger sieve is bounding the size of the set $\mcA\subseteq[M,M+N]$, where each $a\in\mcA$ satisfies

\begin{equation}\label{oldcon}
    \sum_{\substack{q\in Q \\ a\thinspace(\mathrm{mod}\thinspace q) \in\Omega_q}}\!\Lambda(q)~=~\sum_{q\in Q}\Lambda(q).
\end{equation}
That is, each $a$ belongs to the set $\Omega_q$ for every $q\in Q$. However, for our purposes we require a slightly more flexible ``positive proportion" condition, whereby each $a\in\mcA$ satisfies  
\begin{equation}\label{newcon}
    \sum_{q\in Q \atop a \thinspace(\mathrm{mod}\thinspace q) \in\Omega_q}\Lambda(q)~\geq~ c\sum_{q\in Q}\Lambda(q)
\end{equation}
for a fixed choice of $c\in(0,1]$. Fortunately, Gallagher's proof of the larger sieve is short and elementary, so it is not too arduous to incorporate the condition~\eqref{newcon} into the proof. In particular, with the condition~\eqref{newcon}, we use an extra application of Cauchy--Schwarz in Gallagher's proof, yielding the following result.  

\begin{proposition}\label{newlargersieve}
Let $M\in\mathbb{Z}$, $N\in\N$ and $Q$ be a finite set of prime powers. For each $q\in Q$, let $\Omega_q$ be a nonempty set of residue classes modulo $q$ and $\nu(q)=\#\Omega_q$. Let $c \in (0, 1]$ and $\mathcal{A}$ be the set of integers $a \in [M, M+N]$ satisfying 
\begin{equation}\label{largercon}
    \sum_{q\in Q \atop a \thinspace (\mathrm{mod}\thinspace q) \in\Omega_q}\Lambda(q)~\geq~ c\sum_{q\in Q}\Lambda(q).
\end{equation}
Then,
\begin{equation}\label{mainlargeeq}
    \#\mathcal{A} \leq \left(\sum_{q\in Q}\Lambda(q)-\log N\right)/\left(\frac{c^2\big(\sum_{q\in  Q} \Lambda(q)\big)^2}{\sum_{q\in  Q} \nu(q)\Lambda(q)}-\log N\right),
\end{equation}
provided the denominator is positive.
\end{proposition}

Before proving Proposition~\ref{newlargersieve} we make a couple of small remarks. Firstly, we note that the bound for $c=1$ in~\eqref{mainlargeeq} is slightly weaker than the bound~\eqref{orginallargebound} given in Gallagher's larger sieve. This is due to the extra application of Cauchy--Schwarz needed in our proof. Secondly, we note that there are several other variants of the larger sieve appearing in the literature, see e.g.~\cite{croot2004} and~\cite{shao2016}. However, we are not aware of any existing variants with the positive proportion condition~\eqref{largercon}. 

\begin{proof}[Proof of Proposition~\ref{newlargersieve}]
    Throughout the proof, we set, for any integer $r$ (mod $q$),
    \begin{equation*}
        Z(r,q):=\sum_{\substack{a\in\mcA\\a\equiv r\thinspace(\mathrm{mod}\thinspace q)}}\!1\quad\text{and}\quad Z_{\Omega}(q):=\sum_{r\in\Omega_q}Z(r,q).
    \end{equation*}
    We begin by noting that
    \begin{align*}
        \sum_{q\in Q}Z_{\Omega}(q)\Lambda(q)=\sum_{q\in Q}\sum_{\substack{a\in\mcA\\a\thinspace(\mathrm{mod}\thinspace q)\in\Omega_q}}\!\Lambda(q)&=\sum_{a\in\mcA}\sum_{\substack{q\in Q\\ a\thinspace(\mathrm{mod}\thinspace q)\in\Omega_q}}\Lambda(q)\\
        &\geq c\sum_{a\in\mcA}\sum_{q\in Q}\Lambda(q)\\
        &=c\:\# \mcA\sum_{q\in Q}\Lambda(q),
    \end{align*}
    where we have made use of the condition~\eqref{largercon}. Hence, by Cauchy--Schwarz,
    \begin{align}
        c^2(\#\mcA)^2\left(\sum_{q\in Q}\Lambda(q)\right)^2&\leq\left(\sum_{q\in Q}Z_{\Omega}(q)\Lambda(q)\right)^2\notag\\
        &\leq\left(\sum_{q\in Q}\frac{Z_{\Omega}(q)^2}{\nu(q)}\Lambda(q)\right)\left(\sum_{q\in Q}\nu(q)\Lambda(q)\right).\label{cauchy1eq}
    \end{align}
    Notably, the first factor in~\eqref{cauchy1eq} can be bounded from above by following the original proof of Gallagher~\cite{gallagher1971}. In particular, again applying Cauchy--Schwarz,
    \begin{equation*}
        Z_{\Omega}(q)^2\leq \nu(q)\sum_{r\in\Omega_q}Z(r,q)^2\leq\nu(q)\sum_{r\thinspace{(\mathrm{mod}\thinspace q)}}Z(r,q)^2.
    \end{equation*}
    Then, expanding the square and separating into diagonal and off-diagonal terms,
    \begin{equation*}
        \frac{Z_{\Omega}(q)^2}{\nu(q)}\leq\sum_{r\thinspace{(\mathrm{mod}\thinspace q)}}Z(r,q)^2=\sum_{\substack{a,b\in\mcA\\a\equiv b\thinspace(\mathrm{mod}\thinspace q)}}1\leq \#\mcA+\sum_{\substack{a,b\in \mcA\\ a\neq b\\ q\mid a-b}} 1.
    \end{equation*}
    Finally, multiplying by $\Lambda(q)$ then summing over $q\in Q$,
    \begin{align}\label{upperzbound}
        \sum_{q\in Q}\frac{Z_{\Omega}(q)^2}{\nu(q)}\Lambda(q)&\leq \#\mcA\sum_{q\in Q}\Lambda(q)+\sum_{\substack{a,b\in \mcA\\ a\neq b}}\sum_{\substack{q\in Q\\ q\mid a-b}}\Lambda(q)\notag\\
        &\leq\#\mcA\sum_{q\in Q}\Lambda(q)+\left((\#\mcA)^2-\#\mcA\right)\log N
    \end{align}
    where in the last step we have used the identity $\sum_{q\mid n}\Lambda(q)=\log n$. 
    Combining~\eqref{cauchy1eq} and~\eqref{upperzbound} we thus have
    \begin{equation*}
        \frac{c^2(\#\mcA)^2\left(\sum_{q\in Q}\Lambda(q)\right)^2}{\sum_{q\in Q}\nu(q)\Lambda(q)}\leq \#\mcA\sum_{q\in Q}\Lambda(q)+\left((\#\mcA)^2-\#\mcA\right)\log N,
    \end{equation*}
    which rearranges to the desired result.
\end{proof}

\section{Proof of Theorem~\ref{conthm}}\label{consect}
In this section we prove Theorem~\ref{conthm} by incorporating our variant of the larger sieve (Proposition~\ref{newlargersieve}) into the approaches of Hooley~\cite{hooley1967}, and Fan and Pollack~\cite{Fan2025}.

\subsection{Setting up the argument}\label{setupsubsect}
We begin by defining the \emph{$\ell$-test}, which is used to structure the overall argument.

\begin{definition}[$\ell$-test]\label{elltestdef}
    Let $g$ be an integer and $p$ be a prime not dividing $g$. For a prime number $\ell$, we say that $p$ fails the $\ell$-test (with respect to $g$) if
    \begin{equation*}
        p\equiv 1\pmod{\ell}\quad\text{and}\quad g^{(p-1)/\ell}\equiv 1\pmod{p};
    \end{equation*}
    otherwise, we say $p$ passes the $\ell$-test.
\end{definition}
So, by definition, $g$ is a primitive root modulo $p$ precisely when $p$ passes the $\ell$-test for all primes $\ell$. Now in what follows, we let $N$ be a sufficiently large integer and
\begin{equation*}
    x:=(\log N)^B
\end{equation*}
for a fixed\footnote{In Theorem~\ref{conthm}, the case $B\geq 4$ is trivially covered by the case of smaller $B$.} $B\in (2+2e^{-1/3},4)$. Following Fan and Pollack~\cite[\S~4]{Fan2025}, we also define
\begin{equation*}
    W:=\prod_{2<p\leq \log x}p,
\end{equation*}
and
\begin{equation}\label{Sgdef}
    \mcS_g=\mcS_g(x):=\{p\leq x:(g/p)=-1\ \text{and}\ \gcd(p-1,W)=1\}.
\end{equation}
The approach is then to show for almost all $g\in\G$ with  $|g|\in [N,2N]$ (with a small exceptional set) that there exists at least one $p\in\mcS_g$ which passes the $\ell$-test for all primes $\ell$. With this in mind, we note that it suffices to consider $\ell>\log x$ owing to the conditions $(g/p)=-1$ and $\gcd(p-1,W)=1$ in the definition~\eqref{Sgdef} of $\mcS_g$. We now divide the interval $(\log x, \infty)$ into three subintervals as follows:
\begin{equation*}
    I_1 = (\log x, y_1 ], \phantom{m} I_2 = ( y_1, x^{\alpha}] \phantom{m} \text{and}\phantom{m} I_3 = (x^{\alpha}, \infty),
\end{equation*}
where 
\begin{equation*}
    y_1=\frac{x^{\frac{1}{2}-\frac{1}{B}}}{(\log x)^2}
\end{equation*}
and $\alpha$ is a positive real number satisfying
\begin{equation}\label{alphabounds}
    \frac{1}{B}<\alpha<\left(\frac{1}{2}-\frac{1}{B}\right)e^{1/3}.
\end{equation}
Note that such an $\alpha$ must exist by the restriction that $B>2+2e^{-1/3}$. Next, for $j \in \{ 1, 2, 3\}$ and $g\in\G$, let
\begin{equation}\label{Sgjdef}
    \mcS_{g,j}:=\big\{p\in \mcS_g:p\ \text{fails the}\ \ell\text{-test for some}\ \ell\in I_j  \big\},
\end{equation}
where each $\ell$-test is with respect to $g$. So, with this notation
\begin{equation}\label{mastereq}
    \#\{p\in\mcS_g:g\ \text{is a primitive root mod}\ p\}\geq\#\mcS_g-\sum_{j=1}^3\#\mcS_{g,j}.
\end{equation}
Therefore, to prove Theorem~\ref{conthm}, we must show that the right-hand side of~\eqref{mastereq} is positive for all $g\in\G$ with $|g|\in[N,2N]$, with at most $O((\log N)^{B-1})$ exceptions.

\subsection{Bounding each term}
With the setup outlined in the previous section, we now bound the size of the sets $\mcS_g$ and $\mcS_{g,j}$ defined in~\eqref{Sgdef} and~\eqref{Sgjdef} respectively. For $\mcS_g$, $\mcS_{g,1}$ and $\mcS_{g,2}$, this is done uniformly for all $g\in\G$, following~\cite[\S~4]{Fan2025}. Then, when bounding $\#\mcS_{g,3}$ we apply our larger sieve (Proposition~\ref{newlargersieve}) and remove a small exceptional set of $g\in \G$.

To begin with, we state the known asymptotic formula for $\#\mcS_g$, proven using inclusion-exclusion and GRH estimates for character sums and primes in arithmetic progressions.

\begin{lemma}\label{sg0lem}
    Assuming GRH, one has, for all $g\in\G$,
    \begin{equation}\label{Sgasymp}
        \#\mcS_g=\frac{{\msA}_0(x){\msA}_1(g)}{2}\li(x)+O\left(x^{\frac{1}{2}+\frac{1}{B}+o(1)}\right)
    \end{equation}
    where
    \begin{equation}\label{A0def}
        {\msA_0}(x):=\prod_{2<p\leq \log x}\left(1-\frac{1}{p-1}\right)=\left(1+O\left(\frac{1}{\log\log x}\right)\right)\frac{2C_2e^{-\gamma}}{\log\log x},
    \end{equation}
    with $C_2=0.6601\ldots$ the twin-prime constant, and
    \begin{equation}\label{A1def}
        \msA_1(g):=1-(-1)^{\omega(g_1)}1_{4\mid g_1-1,\:g_1\mid W}\prod_{p\mid g_1}\frac{1}{p-2},
    \end{equation}
    with $g_1$ is the square-free integer such that $g=g_1m^2$ for some $m\in\mathbb{N}$.
\end{lemma}
\begin{proof}
    See~\cite[pp.~13--15]{Fan2025}. Note that we have corrected a minor typo in~\cite{Fan2025} whereby one needs a factor of $\log(\Delta dx)$ rather than $\log(dx)$ in the error term introduced at bottom of page 14 of~\cite{Fan2025}. This introduces the factor of $x^{1/B}$ in the error term of~\eqref{Sgasymp}. However, since we (as well as Fan and Pollack) take $B>2$, this correction does not affect the asymptotic for $\#\mcS_g$.
\end{proof}

Next, we recall the known upper bound for $\#\mcS_{g,1}$, which is asymptotically small compared to $\#\mcS_g$. This bound is proven using the Chebotarev density theorem (conditional on GRH) and Brun's sieve.

\begin{lemma}\label{sg1lem}
    Assuming GRH, one has, for all $g\in\G$,
    \begin{equation*}
        \#\mcS_{g,1}\ll\frac{\li(x)}{(\log\log x)^2}=o(\#\mcS_g).
    \end{equation*}
\end{lemma}
\begin{proof}
    See~\cite[p.~16]{Fan2025}.
\end{proof}

Computing an upper bound for $\#\mcS_{g,2}$ is also a standard procedure (cf.~\cite[pp.~16--17]{Fan2025}). However, because the exact interval $I_2$ does not appear in the literature, we give a full proof below for completeness.

\begin{lemma}\label{sg2lem}
    Assuming GRH, one has, for every $g\in\G$,
    \begin{equation*}
        \#\mcS_{g,2}~\leq~\log\left(\frac{\alpha}{\frac{1}{2}-\frac{1}{B}}\right)\msA_0(x)\li(x)+O\left(\frac{\li(x)}{\log x}\right)
    \end{equation*}
    with $\msA_0(x)$ as in~\eqref{A0def}.
\end{lemma}
\begin{proof}
    We first note that
    \begin{equation*}
        \#\mcS_{g,2}\leq\#\{p\leq x:(p-1,W)=1\ \text{and}\ p\equiv 1\ (\text{mod}\ \ell)\ \text{for some}\ \ell\in I_2\}.
    \end{equation*}
    Now, for a \emph{fixed} $\ell\in I_2$,
    \begin{equation*}
        \sum_{\substack{p\leq x\\ p\equiv 1\thinspace(\text{mod}\thinspace\ell)\\(p-1,W)=1}}\! 1=\sum_{\substack{p\leq x\\p\equiv 1\thinspace(\text{mod}\thinspace\ell)}}\sum_{\substack{d\mid W\\ d\mid p-1}}\mu(d)=\sum_{d\mid W}\mu(d)\pi(x;\ell d,1).
    \end{equation*}
    From here, we apply the GRH estimate for primes in arithmetic progressions (see e.g.~\cite[Corollary 13.8]{montyvaughan2007}):
    \begin{equation*}
        \pi(x;\ell d,1)=\frac{\li(x)}{\varphi(\ell d)}+O(x^{1/2}\log x).
    \end{equation*}
    That is,
    \begin{equation*}
        \sum_{\substack{p\leq x\\ p\equiv 1\thinspace(\text{mod}\thinspace\ell)\\(p-1,W)=1}}\!1=\sum_{d\mid W}\mu(d)\left(\frac{\li(x)}{\varphi(\ell d)}+O(x^{1/2}\log x)\right)=\frac{\msA_0(x)\li (x)}{\varphi(\ell)}+O(2^{\pi(\log x)}x^{1/2}\log x).
    \end{equation*}
    We now sum over all $\ell\in I_2$, giving
    \begin{align*}
        \#\mcS_{g,2}&\leq\sum_{\ell\in I_2}\sum_{\substack{p\leq x\\ p\equiv 1\thinspace(\text{mod}\thinspace\ell)\\(p-1,W)=1}}1=\msA_0(x)\li(x)\sum_{y_1<\ell\leq x^{\alpha}}\frac{1}{\ell-1}+O(\pi(x^{\alpha})x^{1/2+o(1)}).
    \end{align*}
    From here, Mertens' theorem and the prime number theorem yield
    \begin{align*}
        \#\mcS_{g,2}&\leq\msA_{0}(x)\li(x)\left(\log\left(\frac{\log x^{\alpha}}{\log y_1}\right)+O\left(\frac{1}{\log y_1}\right)\right)+O\left(x^{1/2+\alpha+o(1)}\right)\\
        &=\log\left(\frac{\alpha}{\frac{1}{2}-\frac{1}{B}}\right)\msA_0(x)\li(x)+O\left(\frac{\li(x)}{\log x}\right),
    \end{align*}
    as required.
\end{proof}

We now move onto the final estimate, that of $\#\mcS_{g,3}$. It is here that we deviate substantially from Fan and Pollack~\cite[\S 4]{Fan2025} by applying Proposition~\ref{newlargersieve} instead of trying to obtain a uniform result.

\begin{lemma}\label{sg3lem}
    Assume GRH and let $\epsilon\in(0,1)$ be fixed. Then, every $g\in\G$ with $|g|\in[N,2N]$ satisfies
    \begin{equation*}
        \#\mcS_{g,3}\leq \epsilon\msA_0(x)\li (x)
    \end{equation*}
    with at most $O((\log N)^{B-1})$ exceptions.
\end{lemma}
\begin{proof}
    Let $g\in\G$ with $|g|\in[N,2N]$. We begin by defining
    \begin{equation*}
        \mcT:=\{p\in(\sqrt{x},x]:(p-1,W)=1\}
    \end{equation*}
    and
    \begin{equation}\label{Qgdef}
        \mcT_{g,3}:=\big\{p\in \mcT: p\nmid g\ \text{and}\ p\ \text{fails the}\ \ell\text{-test for some}\ \ell\in I_3  \big\}
    \end{equation}
    as easier sets to work with than $\mcS_g$ and $\mcS_{g,3}$ (notably $\mcT$ is independent of $g$). Here, we note that
    \begin{equation*}
        \#\mcS_{g,3}\leq \#\mcT_{g,3}+\sqrt{x}
    \end{equation*}
    and, arguing as in the proof of Lemma~\ref{sg2lem},
    \begin{equation*}
       \#\mcT=\sum_{\substack{\sqrt{x}<p\leq x\\(p-1,W)=1}}\!1=\msA_0(x)\li(x)+O(x^{1/2+o(1)})\asymp\frac{x}{\log x\log\log x}.
    \end{equation*} 
    Hence, in what follows, it suffices to show that $\#\mcT_{g,3}<\epsilon\#\mcT$ for all but $O((\log N)^{B-1})$ values of $g\in\G$.
    
    Now, recall that if $p\nmid g$,
    \begin{equation*}
        p\ \text{fails the}\ \ell\text{-test}\ \iff\ \ell\mid p-1\ \text{and}\ g^{(p-1)/\ell}\equiv 1\ \text{(mod}\ p).
    \end{equation*}
    So, with a view to apply Proposition~\ref{newlargersieve}, let $H_{p,\ell}\subseteq(\Z/p\Z)^{\times}$ be the unique subgroup of index $\ell$ and set
    \begin{equation*}
        \Omega_p=\Omega_p(x):=\bigcup_{\substack{\ell\mid p-1\\ \ell>x^{\alpha}}}H_{p,\ell}
    \end{equation*}
    meaning that
    \begin{equation*}
        p\in\mcT_{g,3}\ \iff\ p\in\mcT\ \text{and}\ g\ \text{(mod}\ p)\in\Omega_p.
    \end{equation*}
    Here, we note that, for $p\leq x$,
    \begin{equation}\label{vpbound}
        \nu(p):=\#\Omega_p\leq\sum_{\substack{\ell\mid p-1\\ \ell>x^{\alpha}}}|H_{p,\ell}|=\sum_{\substack{\ell\mid p-1\\\ell>x^{\alpha}}}\frac{p-1}{\ell}<x^{1-\alpha}\sum_{\substack{\ell\mid p-1\\\ell>x^{\alpha}}}1<\frac{x^{1-\alpha}}{\alpha}<Bx^{1-\alpha}.
    \end{equation}
    Next, we define
    \begin{equation*}
        \mcE:=\big\{g\in\G:|g|\in[N,2N]\ \text{and}\ \#\{p\in\mcT:g\thinspace(\text{mod}\thinspace p)\in\Omega_p\}\geq \epsilon\#\mcT\big\}
    \end{equation*}
    so that it now suffices to show $\#\mcE\ll (\log N)^{B-1}$.

    Here, we note that if $\#\{p\in\mcT:g\thinspace(\text{mod}\thinspace p)\in\Omega_p\}\geq \epsilon\#\mcT$, then
    \begin{equation*}
        \sum_{\substack{p\in\mcT\\ g\thinspace(\text{mod}\thinspace p)\in\Omega_p}}\!\Lambda(p)>\frac{\log x}{2}\sum_{\substack{p\in\mcT\\ g\thinspace(\text{mod}\thinspace p)\in\Omega_p}}\!1\geq\frac{\epsilon}{2}\sum_{p\in\mcT}\Lambda(p),
    \end{equation*}
    where we have used that all $p\in\mcT$ satisfy $p\in(\sqrt{x},x]$ and thus ${(\log x)/2<\Lambda(p)\leq \log x}$.

    So, in the notation of Proposition~\ref{newlargersieve}, we set $c=\epsilon/2$ and $\mcA_{+}$ (resp.~ $\mcA_{-}$) to be the set of integers $g\in[N,2N]$ (resp. $g\in[-2N,-N]$) satisfying
    \begin{equation*}
        \sum_{\substack{p\in\mcT\\ g\thinspace(\text{mod}\thinspace p)\in\Omega_p}}\!\Lambda(p)\geq c\sum_{p\in\mcT}\Lambda(p).
    \end{equation*}
    Applying Proposition~\ref{newlargersieve} then yields
    \begin{equation}\label{Ebound}
        \#\mcE\leq (\#\mcA_{+})+(\#\mcA_{-})\leq 2\left(\sum_{p\in\mcT}\Lambda(p)-\log N\right)/\left(\frac{c^2\big(\sum_{p\in\mcT} \Lambda(p)\big)^2}{\sum_{p\in\mcT} \nu(p)\Lambda(p)}-\log N\right).
    \end{equation}
    To bound~\eqref{Ebound} further, we note that
    \begin{equation}\label{Tuppereq}
        \sum_{p\in\mcT}\Lambda(p)\leq \#\mcT\log x\asymp\frac{x}{\log\log x}
    \end{equation}
    and
    \begin{equation}\label{Tlowereq}
        \sum_{p\in\mcT}\Lambda(p)\geq \frac{\#\mcT}{2}\log x\asymp\frac{x}{\log\log x}.
    \end{equation}
    Substituting~\eqref{vpbound},~\eqref{Tuppereq} and~\eqref{Tlowereq} into~\eqref{Ebound}, and using that $\alpha>1/B$ then gives
    \begin{equation*}
        \#\mcE\ll_{B,\epsilon}\frac{x/\log\log x}{x^{\alpha}/\log\log x}=x^{1-\alpha}=(\log N)^{B(1-\alpha)}<(\log N)^{B-1},
    \end{equation*}
    as desired.
\end{proof}

\subsection{Completing the proof}
Combining all of the lemmas from the previous section, we now complete the proof of Theorem~\ref{conthm}.

\begin{proof}[Proof of Theorem~\ref{conthm}]
    Let $N$ be sufficiently large. From~\eqref{mastereq}, we need to show
    \begin{equation*}
        \#\mcS_g-\sum_{j=1}^3\#\mcS_{g,j}>0
    \end{equation*}
    for all $g\in\G$ with $|g|\in[N,2N]$ with at most $O((\log N)^{B-1})$ exceptions. Combining Lemmas~\ref{sg0lem}--\ref{sg3lem}, this is true provided
    \begin{equation}\label{finaleq}
        \frac{{\msA}_1(g)}{2}-\log\left(\frac{\alpha}{\frac{1}{2}-\frac{1}{B}}\right)-\epsilon>0,
    \end{equation}
    with $\msA_{1}(g)$ as defined in~\eqref{A1def} and $0<\epsilon\leq 1$ fixed but arbitrary. Notably, directly from the definition~\eqref{A1def}, we see that
    \begin{equation*}
        \msA_1(g)\geq\msA_{1}(-15)=\frac{2}{3}.
    \end{equation*}
    Hence,~\eqref{finaleq} is satisfied provided
    \begin{equation*}
        \frac{1}{3}-\log\left(\frac{\alpha}{\frac{1}{2}-\frac{1}{B}}\right)-\epsilon>0.
    \end{equation*}
    This always holds true for some choice of $\epsilon>0$ since we chose $\alpha$ to be such that $\alpha<\left(\frac{1}{2}-\frac{1}{B}\right)e^{1/3}$ (see~\eqref{alphabounds}).
\end{proof}

\begin{remark}
    As alluded to in the introduction, it is possible to slightly improve Theorem~\ref{conthm} if one adds additional restrictions on $g$. For example:
    \begin{enumerate}
        \item If we restrict to positive $g\in\G$ then we can take ${B>2+2e^{-2/5}}$$=3.340\ldots$, since $\msA_1(g)\geq \msA_1(21)=4/5$ for $g>1$.
        \item If we restrict to square-free $g\in\G$, then we can take ${B>2+2e^{-1/2}=3.213\ldots}\,$. This is because, in this case, $g=g_1$ and thus $\msA_1(g)=1+o(1)$ as $N\to\infty$. In fact, the same result holds when $g=g_1 m^2$ with $m$ satisfying $m^2 = o(|g|)$, or when $g_1 \not\equiv 1\thinspace(\text{mod}\thinspace 4)$.
    \end{enumerate}
\end{remark}

\section{Proof of Theorem~\ref{unconthm}}\label{unconsect}
The general idea behind the proof of Theorem~\ref{unconthm} is to utilise the following simple fact about primitive roots.
\begin{lemma}\label{simplelem}
    For any prime $p$, there are $\varphi(p-1)$ primitive roots modulo $p$.
\end{lemma}
\begin{proof}
    See e.g.~\cite[Theorem~10.4]{Apostol1976}.
\end{proof}

From Lemma~\ref{simplelem}, we have that $p-\varphi(p-1)$ residue classes modulo $p$ are not primitive roots. So, with a view to count the residue classes which do not correspond to primitive roots for all $p\leq x$, we also state the following simple lemma.

\begin{lemma}\label{fanlem}
    There exists a constant $c>0$ such that
    \begin{equation*}
        \prod_{p\leq x}\frac{p-\varphi(p-1)}{p}=\exp\left(-(c+o(1))\frac{x}{\log x}\right).
    \end{equation*}
\end{lemma}
\begin{proof}
    See~\cite[Lemma~5.1]{Fan2025}.
\end{proof}

We now prove Theorem~\ref{unconthm}.

\begin{proof}[Proof of Theorem~\ref{unconthm}]
    Throughout we may assume that $\Phi(x)=o(\log x)$ since the exceptional set in the theorem will be even smaller for faster growing $\Phi$. To begin with, we define
    \begin{equation*}
        P_N:=\prod_{p\leq\Phi(N)}p
    \end{equation*}
    to be the product of all primes less than or equal to $\Phi(N)$, and
    \begin{equation*}
        \gamma(N):=\prod_{p\leq\Phi(N)}\frac{p-\varphi(p-1)}{p}
    \end{equation*}
    to be the product appearing in Lemma~\ref{fanlem} for $x=\Phi(N)$. Our goal will be to show
    \begin{equation*}
        \#\{g\in[M,M+N]:p_g>\Phi(N)\}=o(N).
    \end{equation*}
    Now, by Lemma~\ref{simplelem} there are $p-\varphi(p-1)$ residue classes which are \emph{not} primitive roots modulo $p$. Therefore, by the Chinese remainder theorem, there are
    \begin{equation*}
        P_N\gamma(N)=\prod_{p\leq\Phi(N)}(p-\varphi(p-1))
    \end{equation*}
    residue classes modulo $P_N$ which are not primitive roots modulo any prime ${p\leq \Phi(N)}$. If $B(N)$ denotes this set of residue classes (so $\#B(N)=P_N\gamma(N)$) then
    \begin{align}
        \#\{g\in[M,M+N]:p_g>\Phi(N)\}&=\sum_{b\in B(N)}\#\{g\in[M,M+N]:g\equiv b\ (\text{mod}\ P_N)\}\notag\\
        &\ll\sum_{b\in B(N)}\left(\frac{N}{P_N}+1\right)\notag\\
        &= \#B(N)\left(\frac{N}{P_N}+1\right)\notag\\
        &=N\gamma(N)+P_N\gamma(N)\label{BNbound}.
    \end{align}
    It thus suffices to show that $N\gamma(N)+P_N\gamma(N)=o(N)$. Firstly, by Lemma~\ref{fanlem}
    \begin{equation*}
        N\gamma(N)=N\exp\left(-(c+o(1))\frac{\Phi(N)}{\log \Phi(N)}\right)=o(N),
    \end{equation*}
    since $\Phi(N)\to\infty$.
    For the second term in~\eqref{BNbound}, we begin by noting that
    \begin{equation*}
        \log(P_N\gamma(N))\leq \log P_N=\sum_{p\leq\Phi(N)}\log p\ll \Phi(N).
    \end{equation*}
    Therefore,
    \begin{equation*}
        P_N\gamma(N)\leq e^{O(\Phi(N))}=o(N),
    \end{equation*}
    where we have made use of our initial assumption $\Phi(N)=o(\log N)$. This completes the proof.
\end{proof}

\section*{Acknowledgements}
Both authors are grateful to the Max Planck Institute for Mathematics in Bonn
for its hospitality and financial support. The second author would like to acknowledge the Indian Institute of Technology (IIT) Palakkad in Kerala for providing an excellent atmosphere to work.

\printbibliography

@article {moree2012,
    AUTHOR = {Moree, P.},
     TITLE = {Artin's primitive root conjecture---a survey},
   JOURNAL = {Integers},
  FJOURNAL = {Integers},
    VOLUME = {12},
      YEAR = {2012},
    NUMBER = {6},
     PAGES = {1305--1416}
}

@article {hooley1967,
    AUTHOR = {Hooley, C.},
     TITLE = {On {A}rtin's conjecture},
   JOURNAL = {J. Reine Angew. Math.},
  FJOURNAL = {Journal f\"ur die Reine und Angewandte Mathematik. [Crelle's
              Journal]},
    VOLUME = {225},
      YEAR = {1967},
     PAGES = {209--220}
}

@article {martin1997,
    AUTHOR = {Martin, G.},
     TITLE = {The least prime primitive root and the shifted sieve},
   JOURNAL = {Acta Arith.},
  FJOURNAL = {Acta Arithmetica},
    VOLUME = {80},
      YEAR = {1997},
    NUMBER = {3},
     PAGES = {277--288}
}

@article {Shoup1992,
    AUTHOR = {Shoup, V.},
     TITLE = {Searching for primitive roots in finite fields},
   JOURNAL = {Math. Comp.},
  FJOURNAL = {Mathematics of Computation},
    VOLUME = {58},
      YEAR = {1992},
    NUMBER = {197},
     PAGES = {369--380}
}

@article {Fan2025,
    AUTHOR = {Fan, K. and Pollack, P.},
     TITLE = {Counting primes with a given primitive root, uniformly},
   JOURNAL = {Mathematika},
  FJOURNAL = {Mathematika. A Journal of Pure and Applied Mathematics},
    VOLUME = {71},
      YEAR = {2025},
    NUMBER = {4},
     PAGES = {Paper No. e70055, 30}
}

@article {gallagher1971,
    AUTHOR = {Gallagher, P. X.},
     TITLE = {A larger sieve},
   JOURNAL = {Acta Arith.},
  FJOURNAL = {Polska Akademia Nauk. Instytut Matematyczny. Acta Arithmetica},
    VOLUME = {18},
      YEAR = {1971},
     PAGES = {77--81}
}

@article {vaughan1973,
    AUTHOR = {Vaughan, R. C.},
     TITLE = {Some applications of {M}ontgomery's sieve},
   JOURNAL = {J. Number Theory},
  FJOURNAL = {Journal of Number Theory},
    VOLUME = {5},
      YEAR = {1973},
     PAGES = {64--79}
}

@article {gallagher1967,
    AUTHOR = {Gallagher, P. X.},
     TITLE = {The large sieve},
   JOURNAL = {Mathematika},
  FJOURNAL = {Mathematika. A Journal of Pure and Applied Mathematics},
    VOLUME = {14},
      YEAR = {1967},
     PAGES = {14--20}
}

@book {Apostol1976,
    AUTHOR = {Apostol, T. M.},
     TITLE = {{I}ntroduction to {A}nalytic {A}umber {T}heory},
    SERIES = {Undergraduate Texts in Mathematics},
 PUBLISHER = {Springer-Verlag, New York-Heidelberg},
      YEAR = {1976},
}

@article {burgess1962,
    AUTHOR = {Burgess, D. A.},
     TITLE = {On character sums and primitive roots},
   JOURNAL = {Proc. London Math. Soc. (3)},
  FJOURNAL = {Proceedings of the London Mathematical Society. Third Series},
    VOLUME = {12},
      YEAR = {1962},
     PAGES = {179--192},
}

@article{martin1998,
    AUTHOR = {Martin, G.},
     TITLE = {Uniform bounds for the least almost-prime primitive root},
   JOURNAL = {Mathematika},
  FJOURNAL = {Mathematika. A Journal of Pure and Applied Mathematics},
    VOLUME = {45},
      YEAR = {1998},
    NUMBER = {1},
     PAGES = {191--207}
}

@article {croot2004,
    AUTHOR = {Croot, III, E. S. and Elsholtz, C.},
     TITLE = {On variants of the larger sieve},
   JOURNAL = {Acta Math. Hungar.},
  FJOURNAL = {Acta Mathematica Hungarica},
    VOLUME = {103},
      YEAR = {2004},
    NUMBER = {3},
     PAGES = {243--254},
}

@article {shao2016,
    AUTHOR = {Shao, X.},
     TITLE = {On an inverse ternary {G}oldbach problem},
   JOURNAL = {Amer. J. Math.},
  FJOURNAL = {American Journal of Mathematics},
    VOLUME = {138},
      YEAR = {2016},
    NUMBER = {5},
     PAGES = {1167--1191}
}

@book {montyvaughan2007,
    AUTHOR = {Montgomery, H. L. and Vaughan, R. C.},
     TITLE = {Multiplicative number theory. {I}. {C}lassical theory},
    SERIES = {Cambridge Studies in Advanced Mathematics},
    VOLUME = {97},
 PUBLISHER = {Cambridge University Press, Cambridge},
      YEAR = {2007}
}
\end{document}